\documentclass[12pt]{amsart}
\usepackage{amsmath,amssymb, xypic,verbatim,amscd,color}
\usepackage{graphicx}
\usepackage{colordvi}
\usepackage{times}

\newcommand\Hom{\operatorname{Hom}}
\newcommand\Sp{\operatorname{Sp}}

\newcommand\frh{\mathfrak{h}}

\newcommand\tensor{\otimes}

\newcommand\codim{\operatorname{codim}}

\newcommand\GL{\operatorname{GL}}\newcommand\LG{\operatorname{LG}}\newcommand\SP{\operatorname{Sp}}

\newcommand\my{\mathcal{Y}}

\newcommand\Ker{\operatorname{Ker}}

\newcommand\Det{\operatorname{Det}}

\newcommand\Gr{\operatorname{Gr}}
\newcommand\SL{\operatorname{SL}}

\newtheorem{theorem}{Theorem}[section]
\newtheorem{remark}[theorem]{Remark}

\newtheorem{corollary}[theorem]{Corollary}

\newtheorem{proposition}[theorem]{Proposition}

\newtheorem{lemma}[theorem]{Lemma}

\newtheorem{definition}[theorem]{Definition}

\theoremstyle{definition}

\newtheorem{example}[theorem]{\bf Example}

  \newcommand{\fb}{\mathfrak{b}}
 
 \newcommand{\fg}{\mathfrak{g}}
 \newcommand{\fh}{\mathfrak{h}}
 \newcommand{\fl}{\mathfrak{l}}
  
 \newcommand{\fp}{\mathfrak{p}}
 
 \newcommand{\fu}{\mathfrak{u}}

 \newcommand{\Y}{\mathfrak{Y}}
  
  \newcommand{\Lam}{\Lambda}
\newcommand{\beqn}{\begin{equation}}
\newcommand{\eeqn}{\end{equation}}

  \newcommand{\al}{\alpha}
  \newcommand{\del}{\delta}  \newcommand{\Del}{\Delta}
  \newcommand{\gam}{\gamma}   
\newcommand{\Exp}{\operatorname{Exp}}

 \newcommand{\cf}{\mathcal{F}}

 \newcommand{\cl}{\mathcal{L}}

\def\ip<#1>{\langle#1\rangle}

\begin{document}
\title[Fulton's conjecture for arbitrary groups]{A generalization of Fulton's
 conjecture for arbitrary groups}

\author{Prakash Belkale}
\author{Shrawan Kumar}
\author{Nicolas Ressayre}
\footnote{P.B. and S.K. were supported by NSF  grants.\\
N.R. was supported by the French National Research Agency (ANR-09-JCJC-0102-01).}
 \maketitle

\begin{abstract} We prove a generalization of Fulton's conjecture which 
relates intersection theory on an arbitrary flag variety to invariant 
theory. \end{abstract}

\section{Introduction} \subsection{The context of Fulton's original 
conjecture}\label{original} Let $L$ be a connected reductive complex 
algebraic group with a Borel subgroup $B_L$ and maximal torus $H\subset 
B_L$. The set of isomorphism
 classes of finite dimensional irreducible representations of $L$ are 
parametrized by the set $X(H)^+$ of $L$-dominant characters of $H$ via the 
highest weight. For $\lambda \in X(H)^+$, let $V(\lambda) = V_L(\lambda)$ 
be the corresponding irreducible representation of $L$ with highest weight 
$\lambda$.
 Define the {\em Littlewood-Richardson coefficients} 
$c^{\nu}_{\lambda,\mu}$ by: $$V(\lambda)\tensor 
V(\mu)=\sum_{\nu}c^{\nu}_{\lambda,\mu} V(\nu).$$

The following result was conjectured by Fulton and proved by 
Knutson-Tao-Woodward [KTW]. (Subsequently, geometric proofs were given by 
Belkale [B$_2$] and Ressayre [R$_2$].) \begin{theorem}\label{1.1} Let
 $L=\GL(r)$ and let $\lambda,\mu,\nu\in X(H)^+$. Then, if 
$c^{\nu}_{\lambda,\mu}=1$, we have $c^{n\nu}_{n\lambda,n\mu}=1$ for every 
positive integer $n$.

(Conversely, if $c^{n\nu}_{n\lambda,n\mu}=1$ for some positive integer $n$, then
$c^{\nu}_{\lambda,\mu}=1$. This follows from the saturation theorem of Knutson-Tao.)
\end{theorem}
Replacing $V(\nu)$ by the dual $V(\nu)^*$, the above theorem is equivalent to the following:
\begin{theorem} Let
 $L=\GL(r)$ and let $\lambda,\mu,\nu\in X(H)^+$. Then,
if $[V(\lambda)\otimes V(\mu)\otimes V(\nu)]^{\SL(r)}=1$,
we have $[V(n\lambda)\otimes V(n\mu)\otimes V(n\nu)]^{\SL(r)}=1$,
 for  every positive integer $n$.
\end{theorem}

The direct generalization of the above theorem for an arbitrary reductive 
$L$ is false (see Example ~\ref{example}(3)). It is also known that the 
saturation theorem fails for arbitrary reductive groups. It is a challenge 
to find an appropriate version of the above result for 
$\operatorname{GL}(r)$ which holds in the setting of general reductive 
groups.

The aim of this paper is to achieve one such generalization. This 
generalization is a relationship between the intersection theory of 
homogeneous spaces and the invariant theory. To obtain this generalization, we must first 
reinterpret the above result for $\GL(r)$ as follows.

Without loss of generality, we only consider the irreducible polynomial representations of
$\GL(r)$. These are parametrized by
the sequences
$\lambda=(\lambda_1\geq \lambda_2\geq \dots \geq \lambda_r\geq 0)$, where we view any such $\lambda$
 as the dominant character
 $\text{diag} (t_1, \dots, t_r)\mapsto t_1^{\lambda_1}\dots  t_r^{\lambda_r}$
 of the standard maximal torus consisting of the diagonal matrices in $\GL(r)$.
Let $\mathfrak{P}(r)$ be the set of such sequences (also called Young diagrams or
partitions) $\lambda=(\lambda_1\geq \lambda_2\geq
\dots \geq \lambda_r\geq 0)$ and let $\mathfrak{P}_k(r)$ be the subset of $\mathfrak{P}(r)$
consisting of those partitions $\lambda$ such that $\lambda_1\leq k$.
Then, the Schubert cells in the  Grassmannian $\Gr(r, r+k)$ of $r$-planes in $\Bbb C^{r+k}$
are parametrized by $\mathfrak{P}_k(r)$ (cf. [F$_2$, $\S$9.4]). For $\lambda \in
\mathfrak{P}_k(r)$, let $\sigma_\lambda$ be the corresponding Schubert cell and
$\bar{\sigma}_\lambda$ its closure. By a classical theorem (cf. loc. cit.), the
structure constants for the intersection product in $H^*(\Gr(r, r+k), \Bbb Z)$ in the basis
$[\bar{\sigma}_\lambda]$ coincide with the corresponding Littlewood-Richardson coefficients
for the representations of $\GL(r)$.
Thus, the above theorem can be reinterpreted as follows:
\begin{theorem} \label{1.3} Let
 $L=\GL(r)$ and let $\lambda,\mu,\nu\in \mathfrak{P}_k(r)$ (for some $k\geq 1$)
 be such that
the intersection product
$$[\bar{\sigma}_\lambda]\cdot [\bar{\sigma}_\mu]\cdot [\bar{\sigma}_\nu]=
[\bar{\sigma}_{\lambda^o}] \,\,\text{in}\, H^*(\Gr(r, r+k), \Bbb Z),$$
where $\lambda^o:=(k\geq \dots\geq k)$ ($r$ copies of $k$). Then,
 $[V(n\lambda) \otimes V(n\mu) \otimes V(n\nu)]^{\SL(r)}=1$, for  every positive integer $n$.
\end{theorem}

\subsection{Generalization for arbitrary groups}

Our generalization of Fulton's conjecture to an arbitrary reductive group
is by considering its equivalent formulation
in Theorem ~\ref{1.3}. Moreover, the generalization replaces the intersection theory of
the  Grassmannians
 by the deformed product $\odot_0$ in the cohomology of $G/P$ introduced in ~\cite{Belkale-kumar}.
 The role of the representation theory of $\SL(r)$ is replaced by the
 representation theory of the semisimple part $L^{ss}$ of the Levi subgroup $L$ of $P$.

To be more precise, let $G$ be a
connected reductive complex algebraic group with a Borel subgroup
$B$ and a maximal torus $H\subset B$. Let $P\supseteq B$ be a (standard)
 parabolic subgroup of $G$. Let $L\supset H$ be the Levi subgroup of $P$,
 $B_L$ the Borel subgroup of $L$ and $L^{ss}=[L,L]$ the semisimple part of $L$.
Let $W$ be the Weyl group of $G$, $W_P$ the Weyl group of $P$, and
 let $W^P$ be the set of minimal
 length coset representatives in $W/W_P$.  For any $w\in W^P$, let $X_w$ be the
 corresponding Schubert
 variety and
$[X_w]\in H^{2(\dim G/P-\ell (w))}(G/P, \Bbb Z)$ the
corresponding Poincar\'e dual class (cf. Section 2).
Also, recall the definition of the deformed product $\odot_0$ in the singular cohomology
$H^*(G/P, \Bbb Z)$ from [BK, Definition 18]. The following is our main theorem
(cf. Theorem ~\ref{main}).

\begin{theorem} \label{1.4} Let $G$ be any connected reductive group and let $P$ be any
standard parabolic subgroup. Then, for any $w_1, \dots, w_s\in W^P$ such that
\beqn \label{e1}[X_{w_1}]\odot_0\dots \odot_0 [X_{w_s}]=[X_e]\in H^{*}(G/P),
\eeqn
we have, for every positive integer $n$,
\beqn \label{e2} \dim \bigl(\bigl[V_L(n\chi_{w_1})\otimes \dots \otimes
V_L(n\chi_{w_s})\bigr]^{L^{ss}}\bigr)=1,
\eeqn
where $V_L(\lambda)$ is the irreducible representation of $L$ with highest weight
 $\lambda$ and $\chi_w$ is defined by the identity ~\eqref{eqn5}.
\end{theorem}

\begin{remark} {\em Let $\mathcal{M}$ be the GIT quotient of
  $(L/B_L)^s$  by the diagonal action of $L^{ss}$ linearized by $\mathcal{L}(\chi_{w_1})
  \boxtimes \dots  \boxtimes \mathcal{L}(\chi_{w_s})$.
 Then, the conclusion of Theorem ~\ref{1.4} is equivalent to the rigidity
 statement that $\mathcal{M}=\text{point}$.  Theorem ~\ref{1.4} can therefore be interpeted as the statement ``multiplicity one in intersection theory leads to
 rigidity in representation theory". }
\end{remark}

Our proof builds upon and further develops the connection between the deformed product
$\odot_0$ and  the representation theory of the Levi subgroup as established in
[BK]. In loc. cit., for any $w\in W^P$,  the line bundle $\mathcal{L}_P(\chi_w)$ on $P/B_L$
 was constructed (see Section ~\ref{march1} for the definitions). Further, the following result
 was proved
 in there (cf. [BK, Corollary 8 and Theorem 15]).
 \begin{proposition}\label{1.6}
 Let $w_1, \dots, w_s\in W^P$ be such that
$$[X_{w_1}]\odot_0\dots \odot_0 [X_{w_s}]=d[X_e]\in H^{*}(G/P),\, \text{for some} \, d\neq 0.
$$ Then,
 $m:=\dim \bigl(H^0\bigl((L/B_L)^s, \bigl(\mathcal{L}_P(\chi_{w_1})\boxtimes \dots  \boxtimes
 \mathcal{L}_P(\chi_{w_s})\bigr)_{|(L/B_L)^s}
 \bigr)^{L^{ss}}\bigr)\neq 0$.
 \end{proposition}
 Note that, by the Borel-Weil theorem, for any $w\in W^P$,
 $H^0(L/B_L,\mathcal{L}_P(\chi_w)_{|(L/B_L)})=V_L(\chi_w)^*$.

The condition ~\eqref{e1} can be translated into the statement that
 a certain map of parameter spaces $X\to Y=(G/B)^s$ appearing in Kleiman's 
theorem is birational. Here $X$ is the
 ``universal intersection'' of Schubert varieties. It is well known that,
 for any birational
 map $X\to Y$ between smooth projective varieties,
 no multiple of the ramification divisor $R$ in $X$ can move even 
infinitesimally
 (i.e., the corresponding Hilbert scheme is reduced, and of dimension $0$ 
at $nR$ for every positive integer $n$).
 We may therefore conclude that $h^0(X,\mathcal{O}(nR))=1$ for every 
positive
 integer $n$. In our situation, $X$ is not smooth, and moreover 
$H^0(X,\mathcal{O}(nR))$ needs to
 be connected to the invariant theory. We overcome these difficulties by 
taking a closer
 look at the codimension one boundary of Schubert varieties.

 The proof also brings into focus the largest (standard) parabolic 
subgroup $Q_w$ acting on a Schubert variety $X_w\subseteq G/P $ (where 
$w\in W^P$), the open $Q_w$ orbit $Y_w\subseteq X_w$ and the smooth locus 
$Z_w\subseteq X_w$. The difference $X_w\setminus Z_w$ is of codimension at 
least two in $X_w$ (since $X_w$ is normal) and can effectively be ignored.

The varieties $Y_w$ give us the link to invariant theory (see Proposition 
~\ref{lemma4.2}). The difference $Z_w\setminus Y_w$ turns out to be quite 
subtle. A key result in the paper
 is that, in the setting of Proposition~\ref{lemma4.2},
the intersection  $\cap_{i}{g_i}Z_{w_i}$ of translates is {\em non-transverse} 
`essentially' at any point which lies in $\bigl(\cap_{i\neq 
j}{g_i}Z_{w_i}\bigr) \cap g_j(Z_{w_j}\setminus Y_{w_j})$ for some $j$ (cf. 
Proposition ~\ref{prop6.1} for a precise statement). This
 reveals the significance of $Q_w$ in the intersection theory of $G/P$ 
and, in particular, to the deformed product $\odot_0$. The ``complexity'' 
of the varieties $Z_w\setminus Y_w$ can therefore be expected to relate to 
the deformed product $\odot_0$.
 Note that by a result of Brion-Polo \cite{Brion-polo}, if $P$ is a 
cominuscule maximal parabolic subgroup, then
 $Y_w= Z_w$, and in this case the deformed cohomology product $\odot_0$ 
coincides with the standard intersection product as well (cf. [BK, Lemma 
19]).

As mentioned above,
 for any cominuscule flag variety $G/P$ (in particular, for the 
Grassmannians
 $\Gr(r, r+k)$),
 the deformed product $\odot_0$ in $H^*(G/P)$ coincides with the standard 
intersection product. In the case of $G=\GL(r+k)$ and $G/P= \Gr(r, r+k)$, 
the set 
$W^P$ can be identified with
 $\mathfrak{P}_k(r)$. For any $\lambda \in W^P$, the corresponding 
irreducible representation of the Levi subgroup $L=\GL(r)\times \GL(k)$ 
with the highest weight $\chi_\lambda$ coincides with 
$V(\lambda)^*\boxtimes V(\tilde{\lambda})$ (cf. [B$_1$]), where 
$V(\lambda)$ is the irreducible representation of $\GL(r)$ as defined in 
Section ~\ref{original} and $\tilde{\lambda}$ is the conjugate partition 
giving rise to the irreducible representation $V(\tilde{\lambda})$ of 
$\GL(k)$. Thus, if we specialize Theorem ~\ref{1.4} to $G=\GL(r+k)$, we 
get Theorem ~\ref{1.3}.

Observe that in the case $G=\GL(r+k)$ and $G/P=\Gr(r,r+k)$, under the 
assumption of Proposition ~\ref{1.6}, from the above discussion and the 
discussion in Section 1.1, we get
 the
 stronger relation $m=d^2$. In general, however, there are no known 
numerical relations between $m$ and $d$ (cf. Examples ~\ref{example}).

We remark that if we replace the condition ~\eqref{e1} in Theorem 
~\ref{1.4}
 by the standard cohomology product, then the conclusion of the theorem is 
false in general (see
 Example ~\ref{example}(4)). Also, the converse to Theorem ~\ref{1.4} is 
not true
 in general (cf.
 Example ~\ref{example}(1)).

\section{Notation} Let $G$ be a connected reductive complex algebraic 
group. We choose a Borel subgroup $B$ and a maximal torus $H\subset B$ and 
let $W=W_G:=N_G(H)/H$ be the associated Weyl group,
 where $N_G(H)$ is the normalizer of $H$ in $G$.  Let $P\supseteq B$ be a 
(standard)
 parabolic subgroup of $G$ and let $U=U_P$ be its unipotent radical. 
Consider the
 Levi subgroup $L=L_P$ of $P$ containing $H$, so that $P$ is the 
semi-direct product
 of $U$ and $L$. Then, $B_L:=B\cap L$ is a Borel subgroup of $L$. Let 
$X(H)$ denote the character group of $H$, i.e., the group of all the 
algebraic group morphisms $H \to \Bbb G_m$. Then, $B_L$ being the 
semidirect product of its commutator $[B_L,B_L]$ and $H$, any $\lambda \in 
X(H)$ extends uniquely to a character of $B_L$. We denote the Lie algebras 
of $G,B,H,P,U,L,B_L$ by the corresponding Gothic characters: 
$\fg,\fb,\fh,\fp,\fu,\fl,\fb_L$ respectively.  Let $R=R_\fg$ be the set of 
roots of $\fg$ with respect to the Cartan subalgebra $\fh$ and let $R^+$ 
be the set of positive roots (i.e., the set of roots of $\fb$). Similarly, 
let $R_\fl$ be the set of roots of $\fl$ with respect to $\fh$ and 
$R_\fl^+$ be the set of roots of $\fb_L$. Let $\Delta = \{\alpha_1, \dots, 
\alpha_\ell\} \subset R^+$ be the set of simple roots, where $\ell$ is the 
semisimple rank of $G$ (i.e., the dimension of $\fh':=\fh\cap [\fg,\fg]$). 
We denote by $\Delta(P)$ the set of simple roots contained in $R_\fl$.  
For any $ 1\leq j\leq \ell$, define the element $x_j\in \fh'$ by 
\begin{equation}\label{eqn0}\alpha_i(x_{j})=\delta_{i,j},\text{ 
}\forall\text{ } 1\leq i\leq \ell. \end{equation}

Recall that if $W_P$ is the Weyl group of $P$ (which is, by definition, 
the Weyl Group $W_L$ of $L$;
 thus $W_P:=W_L$), then in each coset of $W/W_P$ we have a unique member 
$w$ of minimal length.
 This satisfies (cf. [K, Exercise 1.3.E]): \begin{equation}\label{eqn1} 
wB_L w^{-1} \subseteq B. \end{equation} Let $W^P$ be the set of minimal 
length representatives in the cosets of $W/W_P$.

For any $w\in W^P$, define the Schubert cell: \[ C_w =C_w^P:= BwP/P 
\subset G/P.
  \] Then, it is a locally-closed subvariety of $G/P$ isomorphic to the 
affine space $\Bbb A^{\ell(w)}, \ell(w)$ being the length of $w$ (cf. [J, 
Part II, Chapter 13]). Its closure is denoted by $X_w=X^P_w$, which is an 
irreducible (projective) subvariety of $G/P$ of dimension $\ell(w)$. We 
denote the point $wP\in C_w$ by $\dot{w}$.

We also need the shifted Schubert cell: \[ \Lambda_w=\Lambda^P_w := w^{-1} 
BwP/P \subset G/P.
  \]

Let $\mu(X_w)$ denote the fundamental class of $X_w$ considered as an 
element of the singular homology with integral coefficients 
$H_{2\ell(w)}(G/P, \Bbb Z)$ of $G/P$. Then, from the Bruhat decomposition, 
the elements $\{\mu(X_w)\}_{w\in W^P}$ form a $\Bbb Z$-basis of
 $H_*(G/P, \Bbb Z)$. Let $\{[X_w]\}_{w\in W^P}$ be the Poincar\'e dual 
basis of the singular cohomology with integral coefficients $ H^*(G/P, 
\Bbb Z)$. Thus, $[X_w]\in H^{2(\dim G/P-\ell (w))}(G/P, \Bbb Z)$.

The tangent space $T^P=T_{\dot{e}}(G/P)$ of $G/P$ at $e\in G/P$ carries a
canonical  action of $P$ induced from the left multiplication of $P$ on $G/P$.

  We recall the following definition from [BK, Definition 4].

\begin{definition}\label{D1} Fix a positive integer $s\geq 1$.
Let $w_1, \dots, w_s\in W^P$ be such that
  \begin{equation}\label{dim0}
\sum_{j=1}^s\codim \Lam_{w_j}  = \dim G/P.
  \end{equation}
This of course is equivalent to the condition:
\begin{equation}\label{dim0'}
\sum_{j=1}^s\, \ell(w_j)
 = (s-1) \dim G/P.
 \end{equation}
We then call the $s$-tuple $(w_1, \dots, w_s)$ {\it Levi-movable} (for short $L$-{\it movable})
if, for generic $(l_1, \dots, l_s)\in L^s$, the intersection
$l_1\Lambda_{w_1}\cap \cdots \cap l_s\Lambda_{w_s}$ is transverse at $\dot{e}$.
\end{definition}

All the schemes are considered over the base field of complex numbers $\Bbb C$. The varieties
are reduced (but not necessarily irreducible) schemes.
\section{A crucial geometric result}

Let $\pi:X\to Y$ be a regular birational morphism of smooth irreducible varieties with $Y$
projective. Assume that we have a (not necessarily smooth) irreducible projective scheme
$\bar{X}$ containing $X$ as an open subscheme
such that \begin{enumerate} \item the codimension of each irreducible component of
$\bar{X}\setminus X$ in $\bar{X}$ is at least two, \item $\pi$ extends to a regular map
$\bar{\pi}:\bar{X}\to Y$. \end{enumerate}
Let $R$ be the {\it ramification divisor} of $\pi$ in $X$.
It is, by definition, the effective Cartier divisor obtained as the zero scheme of the section of the line bundle $\mathfrak{L}$
 induced by the derivative map $D\pi_x:T_x(X)\to T_{\pi(x)}(Y),$
where the line bundle $\mathfrak{L}$ has base $X$ and fiber $\mathfrak{L}_x$ at any $x\in X$
is given by:
$$\mathfrak{L}_x= \wedge^{\text{top}}(T_x(X)^*)\otimes
\wedge^{\text{top}}(T_{\pi(x)}(Y)).$$
In
the above set up, one has the following crucial result.

\begin{proposition}\label{jan2} For every $n\geq
1$, $h^0(X,\mathcal{O}(nR))=1$,
where $h^0$ denotes the dimension of $H^0$. \end{proposition}

\begin{proof} Clearly
 $\pi_{|X\setminus R}: X\setminus R\to Y$ is an \'{e}tale (and hence quasi-finite)
 birational morphism between
smooth varieties. Hence, by the original form of Zariski's main theorem [M, Chap. III, $\S$9],
it is an open
immersion, i.e., $\pi(X\setminus R)$ is open in $Y$ and $\pi: X\setminus R \to \pi(X\setminus R)$
is an isomorphism. We
will show that $V:=Y\setminus\pi(X\setminus R)$ is of codimension at
least two in $Y$.  This will then imply that
$$H^0(X,\mathcal{O}(nR))\subseteq H^0(X\setminus R,\mathcal{O})=
H^0(Y\setminus V,\mathcal{O})=H^0(Y,\mathcal{O})=\Bbb{C}.$$

 Since $\bar{\pi}$ is surjective, a point $v\in V$ is either in $\bar{\pi}(\bar{X}\setminus X)$, or in
$\pi(R)$, i.e., $V\subseteq \bar{\pi}(\bar{X}\setminus X)\cup \pi(R)$. We show that
$\overline{\pi(R)}$ is of codimension at least two in $Y$ and thus conclude the proof
(by assumption (1)).

To do this let $Z$ be the smallest closed subset of $Y$ so that there exists a morphism
$\sigma: Y\setminus Z\to \bar{X}$  representing the birational inverse to  $\bar{\pi}$.
It is known that the codimension of $Z$ in $Y$ is at least two
(follow [H, Proof of Theorem 8.19 on page 181]). Clearly, $\bar{\pi}\circ \sigma
= I$ on $Y\setminus Z$ and similarly $\sigma\circ\bar{\pi}$ is identity on $\bar{\pi}^{-1}(Y\setminus Z)$ (for the last, note that
 $\sigma\circ\bar{\pi}$ is well defined as a morphism $\bar{\pi}^{-1}(Y\setminus Z)\to \bar{X}$ which
 on an open subset is the identity).
 We therefore find that $\bar{\pi}:\bar{\pi}^{-1}(Y\setminus Z)\to Y\setminus Z$ is an isomorphism.

This tells us that $\bar{\pi}^{-1}(Y\setminus Z)$ is smooth and $\bar{\pi}^{-1}(Y\setminus Z)
\cap R=\emptyset$.
 Hence, $R$ is a subset of $\pi^{-1}(Z)$, or that $\pi(R)\subseteq Z$.
\end{proof}

\section{Some remarks on ramification divisors}\label{motor}

 Consider a linear map $p:V\to W$ between vector spaces of the same dimension.
Let $$\operatorname{Det}(p):=(\wedge^{\text{top}}V)^*\otimes \wedge^{\text{top}}W
={\rm Hom}(\wedge^{\text{top}}V,\wedge^{\text{top}}W).$$
Denote by $\theta(p)$ the canonical element of $\operatorname{Det}(p)$ induced by $p$, i.e.,
$\theta(p)$ is the top exterior power of $p$.
The following lemma is immediate.
\begin{lemma}\label{noise}
Let $p:V\to W$ be as above and $\alpha:W'\to W$ a surjective map. Let $V'\subset V\oplus W'$
consist of $(v,w')$ such that $p(v)=\alpha(w')$ (i.e., $V'$ is the fiber product of
$p$ and $\alpha$).
Let $p':V'\to W'$ be the projection. Then, the kernel of $p'$ is identified
 with the kernel of $p$ via the surjective projection $\pi:V'\to V$. Further, there is a canonical
  isomorphism of the vector spaces $\operatorname{Det}(p)$
 and $\operatorname{Det}(p')$ (defined below), which
 carries $\theta(p)$ to $\theta(p')$. (Observe that $V'$ and $W'$ have the same dimension.)

 Hence, for any fiber diagram of irreducible smooth varieties:
 $$\xymatrix{ X'\ar[d]^{\pi'}\ar[r]^{\hat{f}} & X\ar[d]^\pi\\
   Y'\ar[r]^f & Y,  } $$
   where $f$ is a smooth morphism and $X,Y$ are of the same dimension with $\pi$ a
   dominant morphism, we
   have the following identity between the ramification divisors:
   \beqn \label{4.1} \hat{f}^*(R(\pi))=R(\pi').
   \eeqn
\end{lemma}
The isomorphism $\xi:\Det(p) \to \Det(p')$ is given by:
$$\xi(\theta) (e_1\wedge\dots \wedge e_d\wedge e_{d+1}\wedge \dots \wedge e_n)=
p'(e_1)\wedge \dots \wedge p'(e_d)\wedge \bar{\theta}(\pi(e_{d+1})\wedge \dots \wedge \pi(e_n)),$$
for any $\theta \in \Det(p)=\Hom(\wedge^{\text{top}}V,\wedge^{\text{top}}W),$
where $\{e_1, \dots, e_n\}$ is any basis of $V'$ such that $\{e_1, \dots, e_d\}$ is a
basis of $\text{Ker} (\pi)$ and $\bar{\theta}:=\sigma\circ \theta
$ ($\sigma$ being any section of the map $\wedge^{n-d}(W') \to
\wedge^{n-d}(W)$ induced from $\alpha$). It is easy to see that $\xi$ does not depend upon the choice of the basis
and the section $\sigma$.

\bigskip
Let $X$ be an irreducible smooth variety and $Y_1,\dots, Y_s$ irreducible smooth
locally-closed subvarieties of $X$. Assume that $X$ has a transitive action by a connected
linear algebraic group $G$ and let $G_i$ be algebraic subgroups which keep $Y_i$ stable. Assume
further that $\sum_{i=1}^s\codim(Y_i) =\dim X$.
Let $\Y_i=G\times_{G_i}\, Y_i$ be the total space of the fiber bundle with fiber $Y_i$
associated to the principal $G_i$-bundle $G \to G/G_i$. Then, we have the  morphism $m_i:\Y_i\to X, \,
[g,y_i]\mapsto gy_i$,
where $[g,y_i]$ denotes the equivalence class of
 $(g,y_i)\in G\times Y_i$. Since $Y_i$ is smooth and $G$ acts
 transitively on $X$, by the $G$-equivariance, $m_i$ is a
 smooth morphism (cf. [H, Corollary 10.7, Chap. III]).
Taking their Cartesian product, we get the smooth morphism $m:
\Y_1\times \dots \times \Y_s \to X^s.$ Let $\my$ be the fiber product of $m$ with
 the diagonal map $\delta:X\to X^s$. We get a smooth morphism $\hat{m}:\my\to X$
 by restricting $m$ to
 $\my$. Hence, $\my$ is a smooth and irreducible variety (cf. the proof of Lemma
 ~\ref{lemma3.1}). We also have the morphism $\pi:\my \to G/G_1\times \dots \times G/G_s$
 obtained coordinatewise from the canonical projections $\pi_i:\Y_i\to G/G_i$. For any
 $g_i\in G$ and $y_i\in Y_i$, the map $e_{y_i}:G \to X, \, g \mapsto gy_i,$ induces
 the tangent map $\Psi_{(g_i,y_i)}:T_{g_i}(G)\to T_{g_iy_i}(X).$ Since $Y_i$ is
 $G_i$-stable, this map induces the map $\bar{\Psi}_{(g_i,y_i)}:T_{\bar{g}_i}(G/G_i)
 \to T_{g_iy_i}(X)/ T_{g_iy_i}(g_iY_i),$ where $\bar{g}_i=g_iG_i$. Moreover,
for any $h_i\in G_i$,
\beqn \bar{\Psi}_{(g_i,y_i)}=\bar{\Psi}_{(g_ih_i,h_i^{-1}y_i)}.
\eeqn
To see this, observe that the following diagram is commutative for any
$g_i\in G$ and $h_i\in G_i$.

$$
\xymatrix{
T_{g_i}(G) \ar@{->}[dr]\ar@{->}[rr]^{DR_{h_i}}  & &T_{g_ih_i}(G) \ar@{->}[dl]\\
 & T_{\bar{g}_i}(G/G_i),
  }$$
 where $R_{h_i}: G \to G$ is the right multiplication by $h_i$. Thus,
 $\bar{\Psi}_{(g_i,y_i)}$ depends only upon the equivalence class
 $[g_i,y_i]\in G\times_{G_i}\,Y_i$ and we denote $\bar{\Psi}_{(g_i,y_i)}$ by
 ${\Psi}_{[g_i,y_i]}$. Since $G$ acts transitively on $X$,
 ${\Psi}_{[g_i,y_i]}$ is surjective.

For any $\mathfrak{a}=([g_1,y_1], \dots, [g_s,y_s])\in \my$, we have the following
diagram (for $x=\hat{m}(\mathfrak{a})$):
\beqn \xymatrix {
T_{\mathfrak a}\my \ar[r]^{\hspace {-6em }\pi_{\mathfrak a}}\ar[d]_{D\hat{m}}
& T_{\bar{g}_1}(G/G_1)\oplus \dots
 \oplus T_{\bar{g}_s}(G/G_s) \ar@{->}[d]\\
 T_xX \ar@{->}[r]&\bigoplus_{i=1}^s \frac{T_xX}{T_x(g_iY_i)},
}  \eeqn
   where $\bar{g}_i:=g_iG_i$, the bottom
   horizontal map is the canonical projection in each factor, $D\hat{m}$
    is surjective since $\hat{m}$ is a smooth morphism
   and the right vertical map is the coordinatewise surjective map ${\Psi}_{[g_i,y_i]}$.
\begin{lemma} \label{n4.2} The above diagram is commutative. In fact,
  $T_{\mathfrak a}(\my)$ is the fiber product of $T_x(X)$ and
   $T_{\bar{g}_1}(G/G_1)\oplus \dots
 \oplus T_{\bar{g}_s}(G/G_s)$ via the above diagram.
 \end{lemma}
 \begin{proof}  Let $F$ be the fiber product of  $T_x(X)$ and
   $T_{\bar{g}_1}(G/G_1)\oplus \dots
 \oplus T_{\bar{g}_s}(G/G_s)$. It is easy to see that the above diagram is commutative.
 Moreover, since $y_i=g_i^{-1}x$ for any $\mathfrak{a}=([g_1,y_1], \dots, [g_s,y_s])\in \my$
 with $\hat{m}(\mathfrak{a})=x, T_{\mathfrak a}(\my)$ is a subspace of the fiber
 product $F$. Further,
 \begin{align}\dim \my&=\dim X+\sum_{i=1}^s(\dim \Y_i-\dim X)\\
 &=\dim X+\sum_{i=1}^s(\dim G/G_i+\dim Y_i-\dim X)\\
 &=\dim X+\sum_{i=1}^s(\dim G/G_i-\codim Y_i).
 \end{align}
 From this we see that $\dim F=\dim T_{\mathfrak a}(\my).$ This proves the lemma.
 \end{proof}

\section{Intersection of general translates of Schubert varieties}

We follow the notation from Section 2.  For $w\in W^P$, let $Q_w$ be the stabilizer of
the Schubert variety $X_w$ inside $G/P$ under the left multiplication of $G$ on $G/P$. Then,
clearly, $Q_w$ is a standard parabolic subgroup of $G$.
Let $$Y_w:=Q_w \dot{w}\subset X_w,$$
and let $Z_w$ denote the smooth
locus of $X_w$. Clearly
$$X_w\supset Z_w\supset
Y_w\supset C_w,$$
and each of $Z_w,
Y_w, C_w$ is an open subset of $X_w$.

\begin{remark} {\em It is instructive to look at the example of $G/P=\Gr(r, n)$.
Let
$$F_{\sssize{\bullet}}: 0\subsetneq F_1\subsetneq F_2\subsetneq
\dots\subsetneq F_n=\Bbb{C}^n$$ be the standard flag in $\Bbb{C}^n$, and
let $I=\{i_1<\dots<i_r\}$ be  a subset of $\{1,\dots,n\}$ of cardinality
$r$.
Consider the (closed) Schubert variety
$\Omega_I(F_{\sssize{\bullet}})=\{V\in
\operatorname{Gr}(r, n)\mid \dim V\cap F_{i_k}\geq k, k=1,\dots,
r\}$.
Let $J=\{i\in I: i+1\not\in I\}$.
It is easy to see that $\Omega_I(F_{\sssize{\bullet}})=\{V\in
\operatorname{Gr}(r, n)\mid \dim V\cap F_{i_k}\geq k, \forall
i_k\in J\}$. So $I\setminus J$ is ``redundant'' for the definition
of the closed Schubert variety $\Omega_I(F_{\sssize{\bullet}})$.

It is easy to see that the stabilizer of the Schubert variety
$\Omega_I(F_{\sssize{\bullet}})$ is
$Q_I:=\{g\in \SL(n): gF_j \subset F_j, \forall j\in J\}$. We may think of
$Q_I$ as the set of elements of $\SL(n)$ that preserve the parts of
$F_{\sssize{\bullet}}$ ``essential'' for the definition of the closed
Schubert variety $\Omega_I(F_{\sssize{\bullet}})$. }
\end{remark}

It may be remarked that if $P$ is a minuscule or cominuscule maximal parabloic,
then $Z_w=Y_w$ (cf. [BP]).

Fix a positive integer $s\geq 1$ and fix $w_1, \dots, w_s\in
W^P$, so that
\begin{equation}\label{multiplicity}
[X_{w_1}]\cdot \ldots \cdot [X_{w_s}]=d[X_e]\in H^{*}(G/P),
\,\,\text{for some}\, d>0.
\end{equation}
There are four universal intersections that
will be relevant here. Let $\delta: G/P \to (G/P)^s$ be the diagonal embedding.
 We denote its image by $\delta(G/P)$ and identify it with $G/P$. For a locally-closed
 $B$-subvariety $A\subset G/P$, let
 $\mathfrak{A}:=G\times_B A$ be the total space of the fiber bundle with fiber $A$ associated
 to the principal $B$-bundle $G \to G/B$. Then, there is a $G$-equivariant morphism $m_A:
 \mathfrak{A} \to
 G/P$ defined by $[g,x]\mapsto gx$, which is a smooth morphism if $A$ is smooth.
Now, consider the product
$$\mathfrak{X}:=\mathfrak{X}_{w_1}\times \dots \times \mathfrak{X}_{w_s},$$
where $\mathfrak{X}_{w_i}=G\times_B\,X_{w_i}$,
and similarly define $\mathfrak{Y}, \mathfrak{Z}, \mathfrak{C}$ by replacing
$X_{w_i}$ with $Y_{w_i}, Z_{w_i}, C_{w_i}$ respectively. Let
$m_X: \mathfrak{X} \to (G/P)^s$ be the $G$-equivariant morphism
$m_{X_{w_1}}\times \dots \times m_{X_{w_s}}$ acting componentwise.
We similarly define $m_Y, m_Z,m_C$.

Finally, we define the (universal intersection) $G$-scheme
$ \mathcal{X}$ as the fiber product of $\delta$ with $m_X$.
We similarly define the $G$-schemes $ \mathcal{Y}, \mathcal{Z}, \mathcal{C}$ by replacing
$m_X$ with $m_Y, m_Z,m_C$ respectively. Since $\delta$ is a closed embedding,
$\mathcal{X}, \mathcal{Y}, \mathcal{Z}, \mathcal{C}$ are the scheme theoretic
inverse images of $\delta(G/P)$ under $m_X, m_Y, m_Z,m_C$ respectively. Moreover, since
$m_Y, m_Z,m_C$ are smooth morphisms,
 $ \mathcal{Y}, \mathcal{Z}, \mathcal{C}$ are reduced closed subschemes of
 $\mathfrak{Y}, \mathfrak{Z}, \mathfrak{C}$ respectively.

It is easy to see that (due to the assumption ~\eqref{multiplicity})
\begin{equation}\label{4.2} \dim  \mathcal{X}=s\times \dim (G/B).
\end{equation}

Observe that, set theoretically,
$$ \mathcal{X}=\{(g_1B, \dots, g_sB, x)\in (G/B)^s\times G/P: x\in \cap_{i=1}^s\,g_iX_{w_i}\}.$$
There is a similar description for $\mathfrak{Y}, \mathfrak{Z}, \mathfrak{C}$.

The open embeddings
$$C_{w_i}
\subset Y_{w_i} \subset Z_{w_i} \subset X_{w_i}$$ give rise to $G$-equivariant
open embeddings:
$$\mathcal{C}\subset\mathcal{Y}\subset\mathcal{Z}\subset\mathcal{X},$$
and $\mathcal{X}$ is projective.

\begin{lemma}\label{lemma3.1}
\begin{enumerate}
\item $\mathcal{X}$ is irreducible
and so is  $\mathcal{Y}$, $\mathcal{Z}$  and $ \mathcal{C}$.

\item $\mathcal{Z}$ is a smooth
variety (and hence so is  $\mathcal{Y}$ and $ \mathcal{C}$).
\item
The complement of $\mathcal{Z}$ in $\mathcal{X}$ is of
codimension $\geq 2$.
\end{enumerate}
\end{lemma}

\begin{proof} (1) It is easy to see that each fiber of $m_{X_w}: \mathfrak{X_w} \to G/P$
is irreducible. Thus, each fiber of $m_X: \mathfrak{X} \to (G/P)^s$ is also irreducible. Now, take an
irreducible component $ \mathcal{X}_1$ of $ \mathcal{X}$ such that $ \mathcal{X}_1$
contains the full fiber of
$m_X$ over the base point in $\delta(G/P)$. Since $ \mathcal{X}_1$ is $G$ stable,
$ \mathcal{X}_1$ must contain the full fiber over any point in $\delta(G/P)$. Thus,
$ \mathcal{X}_1= \mathcal{X}$, proving that $ \mathcal{X}$ is irreducible. Since
$\mathcal{Y}$, $\mathcal{Z}$  and $ \mathcal{C}$ are open subsets of $\mathcal{X}$,
they must be irreducible too.

(2) For the second part, observe that the canonical map $\mathcal{Z} \to \delta(G/P)$
is a smooth morphism. Since $G/P$ is smooth, we get the smoothness of $\mathcal{Z}$.

(3) Since the Schubert varieties $X_w$ are normal, the complement of $Z_w$ in ${X}_w$
is of codimension $\geq
2$ and is covered by Schubert cells. Thus, the complement of $\mathfrak{Z}$ in $\mathfrak{X}$
is of codimension $\geq
2$. From this it is easy to see that the complement of $\mathcal{Z}$ in $\mathcal{X}$ is of
codimension $\geq 2$.
\end{proof}

We have a natural $G$-equivariant projection $\pi:\mathcal{X}\to (G/B)^s$ obtained coordinatewise
from the projections $\mathfrak{X}_{w_i}\to G/B.$ As observed in the identity ~\eqref{4.2},
the domain and
the range of $\pi$ have the same dimension. The following lemma follows from
Lemma ~\ref{n4.2}.
\begin{lemma} \label{transverse} For any point $\mathfrak{a}=([g_1,x_1], \dots,
[g_s,x_s])\in \mathcal{Z}$, the
derivative $(D\pi)_{\mathfrak{a}}$ of $\pi$ at $\mathfrak{a}$ has
$$\Ker (D\pi)_{\mathfrak{a}}\simeq \cap_{i=1}^s T_x(g_iZ_{w_i}),$$
where $x=g_1x_1=\dots =g_sx_s$.

In particular, $\pi$ is regular at $\mathfrak{a}$ if and only if the intersection
$\cap_{i=1}^s g_iZ_{w_i}$ in $G/P$ is transverse at $x$.
\end{lemma}


Using  Kleiman's transversality
theorem [BK, Proposition 3] and our assumption ~\eqref{multiplicity}, the map
$\pi_{|\mathcal{Z}}:\mathcal{Z} \to (G/B)^s$ is
generically finite.
Let $R$ be the ramification
divisor for the map $\pi_{|\mathcal{Z}}$ (equipped with the scheme structure
described in Section 3). Under the assumption of the
following corollary, the
hypotheses of Proposition ~\ref{jan2} are in place here and allow us to
conclude the following:
\begin{corollary}\label{cor4.3}
Assume that $d=1$ in equation ~\eqref{multiplicity}. Then,
for every $n\geq 1$,
\begin{equation}\label{conclusion}
h^0(\mathcal{Z},\mathcal{O}(nR))=1.
\end{equation}
\end{corollary}
\begin{proof} By Lemma ~\ref{lemma3.1}, all the hypotheses of Proposition
~\ref{jan2} are satisfied except the hypothesis that $\pi_{|\mathcal{Z}}$ is birational,
 which we now prove.

By [BK, Proposition 3], there exists a nonempty open subset $U\subset (G/B)^s$ such
that for each $x=(g_1B, \dots, g_sB)\in U$, the intersection $\cap_{i=1}^s\,g_iZ_{w_i}$
is transverse at each point of the intersection and $\cap_{i=1}^s\,g_iZ_{w_i}$ is dense in
$\cap_{i=1}^s \,g_iX_{w_i}$. Moreover, since $d=1$ (by assumption), the intersection
$\cap_{i=1}^s\,g_iZ_{w_i}$ consists of a single point. From this we see that
$(\pi_{|\mathcal{Z}})^{-1}(x)$ consists of  exactly one point for each $x\in U$
and, moreover, by Lemma ~\ref{transverse}, $(\pi_{|\mathcal{Z}})^{-1}(x)
\subset \mathcal{Z}\setminus R.$ Thus, $\pi_{|(\pi_{|\mathcal{Z}})^{-1}(U)}:
(\pi_{|\mathcal{Z}})^{-1}(U)\to U$ is an isomorphism, proving that
$\pi_{|\mathcal{Z}}$ is birational. Now applying Proposition ~\ref{jan2}, we get the corollary.
\end{proof}

The aim now is to have equation ~\eqref{conclusion} bear
representation theoretic consequences. However, it is the space
$H^0(\mathcal{Y},\mathcal{O}(nR))$ which has clear relations to
invariant theory.

\section{Connecting $h^0(\mathcal{Y},\mathcal{O}(nR))$ to invariant
theory}\label{march1}

We first prove that $\mathcal{Y}$ and $R\cap \mathcal{Y}$ are obtained from
a base change with connected fibers.
To do this, define
$$\mathfrak{Y}':=(G\times_{Q_{w_1}}\,Y_{w_1})\times \dots \times
(G\times_{Q_{w_s}}\,Y_{w_s}).$$
Similar to the map $m_Y$, we define the map $m_Y': \mathfrak{Y}' \to (G/P)^s$
obtained from the coordinatewise maps $G\times_{Q_{w_i}}\,Y_{w_i} \to G/P,
[g,x]\mapsto gx.$ Again, $m_Y'$ is a smooth morphism. Now, let $ \mathcal{Y}'$ be
the fiber product of $m_Y'$ with $\delta$. Then, $ \mathcal{Y}'$ is an irreducible
 smooth variety of the same dimension as that of $(G/Q_{w_1})\times \dots \times (G/Q_{w_s})$
 (by virtue of the same proof given in the last section for the corresponding results for
 $\mathcal{Y}$).
Similar to the map $\pi_{|\mathcal{Y}} :\mathcal{Y} \to
(G/B)^s$, we have the map
$$\pi' :\mathcal{Y}' \to
(G/Q_{w_1})\times \dots \times (G/Q_{w_s}).$$

It is easy to see that the following diagram is Cartesian:
$$
\xymatrix{ \mathcal{Y}\ar[d]^{\pi}\ar[r] & \mathcal{Y}'\ar[d]^{\pi'}\\
   (G/B)^s\ar[r] &  (G/Q_{w_1})\times \dots \times (G/Q_{w_s}),}$$
where the two horizontal maps are the canonical projections. (To prove this, observe that
the above diagram is clearly Cartesian with $\mathcal{Y}, \mathcal{Y}'$ in the above
diagram   replaced by $\mathfrak{Y}, \mathfrak{Y}' $ respectively.)

Since $\pi$ is a dominant morphism, so is $\pi'$. Thus, by Lemma \ref{noise}, the
ramification divisor $S:=R\cap \mathcal{Y}$ of $\pi|_{\mathcal{Y}}$ is the pull-back
of the
ramification divisor $R'$ of $\pi'$. In particular, the line bundle
$$\mathcal{O}(nR)|_{\mathcal{Y}}=\mathcal{O}(nS).$$

We therefore conclude that under the $G$-equivariant pull-back map,
\begin{lemma} \label{n6.1} For any $n\in \Bbb Z$,
$H^0(\mathcal{Y},\mathcal{O}(nR)|_{\mathcal{Y}})\simeq H^0(\mathcal{Y}',\mathcal{O}(nR'))$,
as $G$-modules.
\end{lemma}

Define the  $P$-variety (under the diagonal action of $P$):
$$\mathcal{P}= (P/(w_1^{-1}Q_{w_1}w_1\cap P))\times \dots
\times (P/(w_s^{-1}Q_{w_s}w_s\cap P)),$$ and define the $G$-equivariant morphism of
$G$-varieties:
$$\phi: G\times_P \mathcal{P}\to \mathcal{Y}', \,\,
[g,(\bar{p}_1, \dots, \bar{p}_s)] \mapsto
([gp_1w_1^{-1}, \dot{w}_1], \dots ,
[gp_sw_s^{-1}, \dot{w}_s]),$$
where $\bar{p_i}=p_i(w_i^{-1}Q_{w_i}w_i\cap P).$

It is easy to see that it is bijective. Since $\mathcal{Y}'$ is smooth and irreducible, $\phi$
 is an isomorphism by [K, Theorem A.11].

For any $w\in W^P$, it is easy to see that the Borel $B_L$ of the Levi subgroup $L$ of
$P$ is contained in $w^{-1} Q_w w\cap L$ (in fact, it is contained
in $w^{-1} B w$ by equation ~\eqref{eqn1}).

For any $\lambda \in X(H)$, we have a  $P$-equivariant
 line bundle $\cl_P (\lambda)$  on $P/B_L$ associated to the principal
 $B_L$-bundle $P\to P/B_L$
via the one dimensional $B_L$-module $\lambda^{-1}$. (As observed in Section 2,
any  $\lambda \in X(H)$ extends uniquely to a character of $B_L$.) The twist
in the definition of  $\cl (\lambda)$ is introduced so that the dominant characters
 correspond to the dominant line bundles.

For $w\in W^P$, define the character $\chi_w\in\frh^*$ by
$$\chi_w=\sum_{\beta\in (R^+\setminus R^+_\fl)\cap w^{-1}R^{+}} \beta \,.$$
 Then, from [K, 1.3.22.3] and equation ~\eqref{eqn1},
\begin{equation}\label{eqn5}
\chi_w = \rho -2\rho^L + w^{-1}\rho ,
\end{equation}
where $\rho$ (resp. $\rho^L$) is half the sum of roots in $R^+$ (resp. in
$ R^+_\fl$). It is easy to see that $\chi_w$ extends as a character of
$w^{-1}Q_ww\cap P$.

\begin{proposition} \label{lemma4.2} Assume that the $s$-tuple $(w_1, \dots, w_s)$ satisfying the condition
~\eqref{multiplicity} is Levi-movable. Then, for any $n\geq 1$,
$$H^0(\mathcal{Y},\mathcal{O}(nR)_{|\mathcal{Y}})^G\simeq \bigl[V_L(n(\chi_{w_1}
-\chi_1))^*\otimes V_L(n\chi_{w_2})^*\otimes \dots \otimes
 V_L(n\chi_{w_s})^*\bigr]^L,$$
 where $V_L(\chi)$ is the irreducible $L$-module with highest weight
 $\chi$. (Observe that for $w\in W^P$, $\chi_w$ is a $L$-dominant weight and so is
 $\chi_w-\chi_1$.)
\end{proposition}
 \begin{proof}
Applying Lemma ~\ref{n4.2} to the case when $X=G/P, Y_i=Y_{w_i}, G_i=Q_{w_i},$
and using the isomorphism $\phi: G\times_P \mathcal{P}\to \mathcal{Y}'$ as above,
we get the following Cartesian diagram (for any $g\in G$ and ${\bf p}=
(\bar{p}_1, \dots, \bar{p}_s)\in \mathcal{P}$):
$$
\xymatrix{ T_{[g,{\bf p}]}(G\times_P\mathcal{P})\ar[d]\ar[r] & \oplus_{i=1}^s
T_{gp_iw_i^{-1}Q_{w_i}}(G/Q_{w_i})\ar[d]^{\pi'}\\
   T_{gP}(G/P)\ar[r] &  \oplus_{i=1}^s
\frac{T_{gP}(G/P)}{T_{gP}(gp_iw_i^{-1}Y_{w_i})},}$$
where the top horizontal map is induced from the $G$-equivariant composite map
$\pi'\circ \phi:G\times_P\mathcal{P}\to \prod_{i=1}^s(G/Q_{w_i})$ and the
bottom horizontal map is the canonical projection in each factor. Thus, by Lemma
~\ref{noise}, the ramification divisor $\phi^{-1}(R')$ is the same as the
 ramification divisor associated to the bundle map (between the vector bundles of the same rank
 over the base space $G\times_P\,\mathcal{P}$):
 $$G\times_P\,(\mathcal{P}\times T^P)\to \bigoplus_{i=1}^s\,G\times_P\,\bigl(P
 \times_{(w_i^{-1}Q_{w_i}w_i\cap P)}\,(T^P/T^P_{w_i})\bigr),$$
 where $T^P$ is the tangent space $T_{\dot{e}}(G/P), T_w^P$ is the tangent space
$T_{\dot{e}}(\Lambda_w), P$ acts diagonally on $\mathcal{P}\times T^P$ and the
map in the $i$-th factor is induced from the composite map
$$\mathcal{P}\times T^P\to (P/(w_i^{-1}Q_{w_i}w_i\cap P))\times T^P \simeq
 P\times_{(w_i^{-1}Q_{w_i}w_i\cap P)}T^P \to  P\times_{(w_i^{-1}Q_{w_i}w_i\cap P)}
 (T^P/T^P_{w_i}).$$
 Thus, by [BK, Lemma 6 and the discussion following it] and Lemma ~\ref{noise}, the line bundle corresponding to the divisor
 $\phi^{-1}(R')$ is $G$-equivariantly isomorphic to the line bundle
 $G\times_P\,\mathcal{M}$ over the base space  $G\times_P\,\mathcal{P}$,
 where $$\mathcal{M}=\cl_P (\chi_{w_1}-\chi_1)\boxtimes \cl_P (\chi_{w_2})\boxtimes
 \dots \boxtimes \cl_P (\chi_{w_s}).$$
 Observe that, for any $w\in W^P$,  the line bundle $\cl_P (\chi_{w})$, though defined on $P/B_L$,
 descends to a line bundle on $ P/(w^{-1}Q_{w}w\cap P)$ since the character $\chi_w$ extends
 to a character of $w^{-1}Q_{w}w\cap P$.
 Thus,
 \begin{align*}
 H^0(Y, \mathcal{O}(nR)_{|Y})^G &\simeq H^0(Y', \mathcal{O}(nR'))^G, \, \text{by Lemma}\,
 ~\ref{n6.1}\\
 &\simeq H^0(G\times_P\,\mathcal{P}, G\times_P\,\mathcal{M}^{\otimes n})^G\\
 & \simeq H^0(\mathcal{P}, \mathcal{M}^{\otimes n})^P\\
 & \simeq H^0(\mathfrak{L}, \mathcal{M}^{\otimes n}_{|\mathfrak{L}})^L,
 \end{align*}
 where
  $$\mathfrak{L}:=(L/(w_1^{-1}Q_{w_1}w_1\cap L))\times \dots
\times (L/(w_s^{-1}Q_{w_s}w_s\cap L))$$ and the last isomorphism follows from
  [BK, Theorem 15 and Remark 31(a)].

 Thus, the proposition follows from the  Borel-Weil theorem.
\end{proof}

\section{Study of codimension one cells in the Schubert varieties}

We continue to follow the notation and assumptions from Section 2. The following lemma can be found in [BP, \S2.6]. However, we
include its proof for completeness.

  \begin{lemma} \label{lemma1} For any $w\in W^P$, the stabilizer $Q_w$ of $X_w$
satisfies
  \beqn  \label{5.1} \Del (Q_w) = \Del_w ,
  \eeqn where $\Del_w := \Del\cap w(R^+_{\fl} \sqcup R^-)$ and
$R^-$ is the set of negative roots of $\fg$.

Thus,
  \beqn  \label{5.2} \Del (Q_w) = \Del \cap (ww_o^P)R^- ,
  \eeqn where $w_o^P$ is the longest element of the Weyl group $W_L$ of
$L$.  (Observe that $ww_o^P$ is the longest element $\hat{w}$ in the coset
$wW_L$.)
  \end{lemma}

  \begin{proof} We first prove equation ~\eqref{5.1}.  Observe that
    \begin{align*} w\Bigl( R^+_{\fl}\sqcup R^-\Bigr) &= w\Bigl(
R^+_{\fl}\sqcup R^-_{\fl}\sqcup (R^-\backslash R^-_{\fl})\Bigr)\\ &=
\hat{w}\Bigl( R_{\fl}\sqcup (R^-\backslash R^-_{\fl})\Bigr) .
  \end{align*} Thus,
  \begin{align} \label{5.3} \Del_w &= \Del\cap \hat{w} \Bigl( R_{\fl}\sqcup
(R^-\backslash R^-_{\fl})\Bigr) \notag \\ &= \Del\cap \hat{w} R^-,
\qquad\text{since } \hat{w}(R^+_{\fl})\subset R^- .
  \end{align} Take $\al_i\in\Del_w = \Del\cap\hat{w}R^-$.  Then,
  \begin{align*} s_i\, BwP/P &\subset (BwP/P) \cup (Bs_iwP/P)\\ &= (BwP/P)
\cup (Bs_i\hat{w}P/P) .
  \end{align*} But $s_i\hat{w} < \hat{w}$ since $(\hat{w})^{-1}\, \al_i\in
R^-$.  Hence,
  \[ s_i X_w \subset X_w .  \] This proves the inclusion $\Del (Q_w)
\supset \Del_w = \Del\cap \hat{w}R^-$.

Conversely, take $\al_i\in\Del (Q_w)$, i.e., $s_i X_w \subset X_w$.
Thus, $s_i\hat{w}<\hat{w}$ and hence $\hat{w}^{-1} \al_i\in R^-$.  This
proves the inclusion $\Del (Q_w)\subset \Del_w$ and hence equation ~\eqref{5.1} is proved.
The  equation ~\eqref{5.2} follows by combining equations ~\eqref{5.1} and ~\eqref{5.3}.
  \end{proof}

  \begin{proposition} \label{prop5.2} Let $v\overset{\beta}{\rightarrow} w \in W^P$ (i.e.,
$v,w\in W^P$, $\beta\in R^+$ such that $w = s_{\beta}v$ and $\ell
(w)=\ell (v)+1$).  Then, the (codimension one) cell $C_v$ of $X_w$ is
contained in
  $Q_w wP/P$ if and only if $ \beta\in\Del_w.$

  In particular, $\beta$ is a simple root in this case.
  \end{proposition}

  \begin{proof} We first prove the implication `$\Leftarrow$': If
$\beta\in\Del_w$, then $\beta\in\Del(Q_w)$, by Lemma \ref{lemma1}.  Thus, $\dot{v}
= s_{\beta}wP \in Q_w wP/P$.

Conversely, we prove the implication `$\Rightarrow$': Assume, if possible,
that $\dot{v} \in Q_wwP/P$ but $\beta\notin\Del_w$.  We first show that
$X_v$ is stable under $Q_w$ (assuming $\beta\notin\Del_w$).  By Lemma
\ref{lemma1}, it suffices to show that for any $\al_j\in\Del_w = \Del\cap\hat{w}
R^-$, we have $\al_j\in\Del_v$.  Since $\hat{w}^{-1}\al_j\in R^-$, we get
$s_j\hat{w} < \hat{w}$.  Take a reduced decomposition $\hat{w} =
s_js_{i_1}\cdots s_{i_d}$.  Since $v\overset{\beta}{\rightarrow}w$, then
so is
$\hat{v}\overset{\beta}{\rightarrow}\hat{w}$.  Hence, there exists a
(unique)
$1\leq p\leq d$ such that $\hat{v}=s_j s_{i_1}\cdots \hat{s}_{i_p} \cdots
s_{i_d}$ and, of course, it is a reduced decomposition.  (Here we have
used the assumption that $\beta\notin\Del_w$.)

Thus, $s_j\hat{v} < \hat{v}$, i.e., $\hat{v}^{-1}\al_j\in R^-$ and hence
$\al_j\in\Del_v$.  This proves the assertion that $X_v$ is stable under
$Q_w$.

By assumption, $\dot{v}\in Q_wwP/P$, i.e., $\dot{v}=q\dot{w}$ for some
$q\in Q_w$.  Thus, $q^{-1}\dot{v}=\dot{w}$ and hence $\dot{w}\in Q_w X_v
= X_v$, which is a contradiction.  This contradiction shows that
$\beta\in\Del_w$ and hence completes the proof of the proposition.
  \end{proof}

For $w\in W^P$, it is easy to see that the tangent space, as an
$H$-module (induced from the left multiplication of $H$ on $X_w$), is given by:
  \beqn \label{5.4} T_{\dot{w}} (X_w) \simeq \bigoplus_{\gam\in R^+\cap wR^-}
\fg_{\gam} ,
  \eeqn
  where $\fg_\gamma$ is the root space of $\fg$ corresponding to the root $\gamma$. Hence,
  \beqn \label{5.5} T_{\dot{e}} \bigl( w^{-1} X_w\bigr) \simeq \bigoplus_{\gam\in
R^-\cap w^{-1}R^+} \fg_{\gam}.
  \eeqn

The following lemma determines the tangent space along codimension one
cells.

  \begin{lemma} \label{lemma5.3} For $v\overset{\beta}{\rightarrow} w \in W^P$, the tangent
space, as an $H$-module, is given by:
  \[ T_{\dot{v}} (X_w) \simeq \Biggl( \bigoplus_{\gam\in R^+\cap vR^-}
\fg_{\gam}\Biggr) \bigoplus \fg_{-\beta} .
  \] Thus, as an $H$-module,
  \[ T_{\dot{e}} (v^{-1} X_w) \simeq \Biggl( \bigoplus_{\gam\in R^-\cap
v^{-1}R^+} \fg_{\gam}\Biggr) \bigoplus \fg_{-v^{-1}\beta} .
  \]

(Observe that $\dot{v}$ is a smooth point of $X_w$ since $X_w$ is
normal; in particular, its singular locus is of codimension at least two.)
  \end{lemma}

  \begin{proof} Since $\dot{v}\in X_v \subset X_w$, by ~\eqref{5.4},
  \beqn \label{5.6} \bigoplus_{\gam\in R^+\cap vR^-} \fg_{\gam}\subset T_{\dot{v}}
(X_w) .
  \eeqn

For any root $\al\in R$, let $U_{\al} := \Exp (\fg_{\al})\subset G$ be the
corresponding 1-dimensional unipotent group.  Then,
  \[ U_{\beta}U_{-\beta}\dot{w} = U_{\beta}wU_{-w^{-1}\beta}\dot{e} =
U_{\beta}\dot{w}\subset X_w \qquad (\text{since }w^{-1}\beta\in R^- ).
  \] Hence, $U_{\beta}HU_{-\beta}\dot{w}\subset X_w$.  But, from the
$SL(2)$-theory, $\overline{U_{\beta}HU_{-\beta}}\supset
U_{-\beta}s_{\beta}H$.  In particular,
  \[ U_{-\beta}s_{\beta}H\dot{w} \subset X_w, \text{ i.e., }
U_{-\beta}\dot{v} \subset X_w .
  \] This proves that
  \beqn \label{5.7} \fg_{-\beta} \subset T_{\dot{v}} (X_w).
  \eeqn

Combining ~\eqref{5.6}--\eqref{5.7}, we get
  \[\bigl( \bigoplus_{\gam\in R^+\cap vR^-} \fg_{\gam}\bigr) \bigoplus
\fg_{-\beta}\subset T_{\dot{v}} (X_w) .
  \] But, both the sides are of the same dimension $\ell (v)+1$, proving
the lemma.
  \end{proof}

As above, let $P$ be any standard parabolic subgroup of $G$ and let $x_P
\in \fh'=\fh\cap [\fg,\fg]$ be the
element defined by \begin{align*}
  \al_i(x_P) &= 0, \qquad \text{for all the simple roots $\al_i\in\Del
(P)$}\\
  &=1, \qquad\text{for all the  simple roots $\al_i\notin \Del (P)$}.
\end{align*}
Then, $x_P$ is in the center of the Lie algebra $\mathfrak{l}$.

Set $m_o=\theta (x_P)$, where $\theta$ is the highest root of $\fg$.
(Observe that $m_o \leq 2$ for any maximal parabolic subgroup $P$ of a
classical group $G$.)  Define a
decomposition of $T_{\dot{e}} (G/P)$ as a direct sum of $L$-submodules as follows.
First decompose $T_{\dot{e}}(G/P)$ as a direct sum of $H$-eigenspaces
(induced from  the canonical action of $H$ on $G/P$):
  \[ T_{\dot{e}} (G/P) = \bigoplus_{\beta\in R^+\backslash R^+_{\fl}}
T_{\dot{e}} (G/P)_{-\beta} .
    \] For any $1\leq j\leq m_o$, define
  \[ V_j = \bigoplus_{ \substack{\beta\in R^+\backslash R^+_{\fl}:\\ \beta
(x_P)=j} } T_{\dot{e}} (G/P)_{-\beta} .
  \] Clearly, each $V_j$ is a $L$-submodule of $T_{\dot{e}}(G/P)$ and we have the
decomposition (as $L$-modules)
  \[ T_{\dot{e}} (G/P) = \bigoplus^{m_o}_{j=1} V_j .  \]

Define an increasing filtration of $T_{\dot{e}}(G/P)$ by $P$-submodules given by
  \[ \cf_1 \subset \cf_2 \subset \cf_3 \subset \cdots \subset \cf_{m_o} =
T_{\dot{e}}(G/P) ,
  \] where
  \[ \cf_d = \bigoplus^d_{j=1} V_j .  \]

  For any subvariety $Z \subset G/P$ such that $\dot{e}$ is a
smooth point of $Z$, define
  \[ V_j(Z) := V_j \cap T_{\dot{e}}(Z) .
  \] Also, we get the increasing filtration $\cf_j(Z)$ of $T_{\dot{e}}(Z)$
given by
  \[ \cf_j(Z) := \cf_j\cap T_{\dot{e}}(Z).
  \] We set, for $1\leq j\leq m_o$,
  \[ d_j(Z) = \text{dimension of } V_j(Z) .
  \] (Observe that $d_0 (Z) =0$ since $T_{\dot{e}}(Z)\subset
T_{\dot{e}}(G/P) \simeq \bigoplus_{\al\in R^+\backslash R^+_{\fl}}
\fg_{-\alpha}$.)

Let $\mathfrak{z}(L)$ be the center of $L$. If $Z$, as above, is  $\mathfrak{z}(L)$-stable,
we get the decomposition (as  $\mathfrak{z}(L)$-modules)
    \[ T_{\dot{e}}(Z) = \bigoplus^{m_o}_{j=1} V_j(Z) .
  \]
  \begin{theorem} \label{degree} For any $v\overset{\beta}{\rightarrow} w$ in $W^P$ such
that $\dot{v}$ is not in the $Q_w$-orbit of $\dot{w}$, there exists
$1\leq j\leq m_o$ such that
    \beqn\label{5.9} d_j(w^{-1} X_w) \neq d_j (v^{-1} X_w).
  \eeqn
    \end{theorem}

  \begin{proof} Let us set $\al := v^{-1}\beta\in R^+$.  For any $1\leq
j\leq m_o$, we get (by Lemma \ref{lemma5.3} and equation ~\eqref{5.5} applied to $v$)
    \beqn \label{5.8} d_j(v^{-1} X_w) = d_j (v^{-1}X_v) + \del_{j,\al (x_P)} .
  \eeqn
  By the equation ~\eqref{5.5}, the roots in $T_{\dot{e}}(w^{-1}X_w)$ are
precisely $R^-\cap w^{-1}R^+$ $\Bigl($i.e., $T_{\dot{e}} (w^{-1}X_w)
\simeq \bigoplus_{\gam\in R^-\cap w^{-1}R^+}\fg_\gamma\Bigr)$.  Set
  \[ \Phi_{w^{-1}} := R^+\cap w^{-1}R^- .
  \] Then, as is well known,
  \[ \sum_{\del\in\Phi_{w^{-1}}} \del = \rho -w^{-1}\rho ,
  \] where $\rho$ is half the sum of all the positive roots.

Thus (abbreviating $d_j(w^{-1}X_w)$ by $d_j$ and $d_j(v^{-1}X_v)$ by
$d'_j$),
  \beqn \label{5.10} (\rho -w^{-1}\rho )(x_P) = d_1 +2d_2 +\cdots + m_od_{m_o} .
  \eeqn Similarly,
   \beqn \label{5.11} (\rho -v^{-1}\rho )(x_P) = d'_1 +2d'_2 +\cdots + m_od'_{m_o} .
  \eeqn Of course,
  \beqn \label{5.12} d_1+d_2+\cdots +d_{m_o} = \ell (w),
  \eeqn and
  \beqn \label{5.13} d'_1+d'_2+\cdots +d'_{m_o} = \ell (v) = \ell (w)-1.
  \eeqn

Now,
  \begin{align} \label{5.14} (\rho -w^{-1}\rho )(x_P) - (\rho -v^{-1}\rho )(x_P) &=
(v^{-1}\rho -w^{-1}\rho )(x_P) \notag\\
 &= (v^{-1}\rho -s_{\al}v^{-1}\rho )(x_P), \text{ since}\,\, w=vs_{\al}\notag\\
 &= \ip<v^{-1}\rho ,\al^\vee>\, \al (x_P)\notag\\
 &= \ip<\rho , (v\al)^\vee>\, \al (x_P)\notag\\
 &= \ip<\rho ,\beta^\vee>\, \al (x_P) .
    \end{align}

    On the other hand, by ~\eqref{5.10}--\eqref{5.13},
  \begin{align} \label{5.15} (\rho -w^{-1}\rho )(x_P) &- (\rho -v^{-1}\rho )(x_P)
\notag \\
 &= (d_1-d'_1) + 2(d_2-d'_2) + \cdots + m_o(d_{m_o}-d'_{m_o}) \notag \\
 &= 1 + (d_2-d'_2) + 2(d_3-d'_3) + \cdots + (m_o-1)(d_{m_o}-d'_{m_o}) .
   \end{align} Combining ~\eqref{5.14}--\eqref{5.15}, we get
  \beqn \label{5.16} 1 + (d_2-d'_2) + 2(d_3-d'_3) + \cdots + (m_o-1)(d_{m_o}-d'_{m_o})
= \ip<\rho ,\beta^\vee>\, \al (x_P) .
  \eeqn

If ~\eqref{5.9} were false, we would get
  \[ d_j = d_j( v^{-1} X_w), \qquad\text{ for all } 1\leq j\leq m_o,
  \] i.e., by the identity ~\eqref{5.8}, we would get
  \[ d_j = d'_j\quad\text{ for all $j\neq \al (x_P)$ and } \quad d_{\al
(x_P)} = d'_{\al (x_P)} +1 .
  \] Combining this with the identity ~\eqref{5.16}, we would get
  \begin{align} \label{5.17} 1 +\al (x_P)-1 &= \ip<\rho ,\beta^\vee>\, \al (x_P),
\text{
i.e., }\notag\\ \al (x_P) &= \ip<\rho ,\beta^\vee>\, \al (x_P).
  \end{align}

But, by the definition of $\beta$, it is easy to see that if $\beta$ were
a simple root, then $\beta\in\Del_w$.  Since, by assumption, $\dot{v}$ is
not in the $Q_w$-orbit of $\dot{w}$, this contradicts Proposition \ref{prop5.2}.
Hence, $\beta$ is not a simple root and this contradicts the identity ~\eqref{5.17}.
(Observe that $\al (x_P)\neq 0$, since $v,w\in W^P$ and $w=vs_{\al}$.)
This contradiction arose because we assumed that ~\eqref{5.9} is false.  This
proves ~\eqref{5.9} and hence the theorem is proved.
  \end{proof}

\section{Main theorem and its proof}

We follow the notation and assumptions from Section 5. In particular, let
$w_1, \dots , w_s\in W^P$ be such that identity
~\eqref{multiplicity} is satisfied for some $d>0$. We assume further that the $s$-tuple
$(w_1, \dots , w_s)$ is Levi-movable. This will be our assumption through this section.

\begin{proposition} \label{prop6.1} Under the above assumption, there exists a closed subset $A$ of
$\mathcal{Z}$ such that
\begin{equation}\label{(H)}
\mathcal{Z}\setminus \mathcal{Y}\subseteq R\cup A,\
\codim(A,\mathcal{Z})\geq 2.
\end{equation}
\end{proposition}
\begin{proof}
 Let $\mathcal{Z}^o:=\mathcal{Z}\setminus R.$ It suffices to show that for
 $u_1, \dots ,  u_s\in W^P$ such that $u_i=w_i$ for all $i\neq i_o$ and $u_{i_o}\to w_{i_o}$
 for some $1\leq i_o\leq s$ and $\dot{u}_{i_o}\notin Y_{w_{i_o}}$,
 $$\mathcal{Z}^o\cap (\mathfrak{C}_{u_1}\times \dots  \times \mathfrak{C}_{u_s})
 =\emptyset.$$

 Since  the $s$-tuple
$(w_1, \dots , w_s)$ is Levi-movable, there exist $l_1, \dots , l_s\in L$ such that
the standard quotient map
$$T_{\dot{e}}(G/P) \to \bigoplus_{i=1}^s \,T_{\dot{e}}(G/P)/T_{\dot{e}}(l_i\Lambda_{w_i})$$
is an isomorphism. Hence, the eigenspaces corresponding to any eigenvalue $1\leq
j\leq m_o$
under the action of $x_P$ also are isomorphic, i.e.,
$$V_j(G/P)\simeq \bigoplus_{i=1}^s\, V_j(G/P)/V_j(l_i\Lambda_{w_i}),$$
where $V_j$ is as in Section 7.
 (Here we have used the fact that $l_i\Lambda_{w_i}$ is $\mathfrak{z}(L)$-stable.)
 In particular, since the filtration $\mathcal{F}_j$ of $T_{\dot{e}}(G/P)$ is $P$-stable,
 for any $p_1,\dots, p_s\in P$,
 \beqn \label{6.1} \dim \mathcal{F}_j=\sum_{i=1}^s (\dim \mathcal{F}_j -
 \dim (\mathcal{F}_j(l_i\Lambda_{w_i}))
 =\sum_{i=1}^s (\dim \mathcal{F}_j -
 \dim (\mathcal{F}_j(p_i\Lambda_{w_i})).
 \eeqn

 If nonempty, take $\mathfrak{a}=([g_1,x_1], \dots , [g_s,x_s])\in
 \mathcal{Z}^o\cap (\mathfrak{C}_{u_1}\times \dots  \times \mathfrak{C}_{u_s})$,
 for $g_i\in G$ and $x_i\in {C}_{u_i}$. In particular,
 $g_1x_1=\dots =g_sx_s.$ Let us denote this common element by $gP$. From the
 $G$-equivariance, we can assume that $g=e$. By Lemma \ref{transverse},
 the quotient map
 $$T_{\dot{e}}(G/P) \to \bigoplus_{i=1}^s \,T_{\dot{e}}(G/P)/T_{\dot{e}}
 (p_iu_i^{-1}Z_{w_i})$$
 is an isomorphism, where $p_i\in P$ is any element chosen such that
 $g_i\in p_iu_i^{-1}B$.
 In particular, for any $j$, the quotient map
  $$\mathcal{F}_j\to \bigoplus_{i=1}^s \mathcal{F}_j/
 (T_{\dot{e}}
 (p_iu_i^{-1}Z_{w_i})\cap\mathcal{F}_j)$$
 is injective.
  Thus, for any $j$,
  \beqn \label{6.2} \dim \mathcal{F}_j\leq\sum_{i=1}^s (\dim \mathcal{F}_j -
 \dim (\mathcal{F}_j(p_iu_i^{-1}Z_{w_i}))
 =\sum_{i=1}^s (\dim \mathcal{F}_j -
 \dim (\mathcal{F}_j(u_i^{-1}Z_{w_i})).
 \eeqn
 Considering the image of $T_{\dot{e}}(g^{-1}Z_w)$ in $T_{\dot{e}}(G/P)/\mathcal{F}_j,$
 for $gP\in Z_w$, it is easy to see that, for any $u,w\in W^P$ such that $\dot{u}\in Z_w$
 and any $j$, we have
 \beqn \label{6.4} \dim \mathcal{F}_j(w^{-1}Z_w) \leq \dim
 \mathcal{F}_j(u^{-1}Z_w).
 \eeqn

 Now, let $j_o$ be an integer such that
 \beqn \label{6.5} \dim \mathcal{F}_{j_o}(w_i^{-1}Z_{w_i}) \neq \dim
 \mathcal{F}_{j_o}(u_i^{-1}Z_{w_i})\,\,\text {for}\, i=i_o.
 \eeqn
This is possible by virtue of Theorem \ref{degree}.
This contradicts the inequality ~\eqref{6.2} for $j=j_o$ (by using
~\eqref{6.1}, ~\eqref{6.4}--\eqref{6.5}).
Hence the proposition is proved.
\end{proof}

Recall the definition of the deformed product $\odot_0$ in the singular cohomology
$H^*(G/P, \Bbb Z)$ from [BK, Definition 18]. We now come to our main theorem.

\begin{theorem} \label{main} Let $G$ be any connected reductive group and let $P$ be any
standard parabolic subgroup. Then, for any $w_1, \dots, w_s\in W^P$ such that
$$[X_{w_1}]\odot_0\dots \odot_0 [X_{w_s}]=[X_e]\in H^{*}(G/P),$$
we have (for any $n\geq 1$)
\beqn \label{7.6} \dim \Hom_L(V_L(n\chi_1), V_L(n\chi_{w_1})\otimes \dots \otimes V_L(n\chi_{w_s}))=1,
\eeqn
where $\chi_w$ is defined by identity ~\eqref{eqn5}. Equivalently, we have (for the commutator
subgroup $L^{ss}:=[L,L]$):
\beqn \label{7.7} \dim \bigl(\bigl[V_L(n\chi_{w_1})\otimes \dots \otimes
 V_L(n\chi_{w_s})\bigr]^{L^{ss}}\bigr)=1,
\,\forall \, n\geq 1.
\eeqn
\end{theorem}
\begin{proof} By [BK, Theorem 15 and Proposition 17], the $s$-tuple $(w_1, \dots, w_s)$
is $L$-movable. Hence,
 by Proposition ~\ref{lemma4.2}, we have
 $$H^0(\mathcal{Y},\mathcal{O}(nR)_{|\mathcal{Y}})^G\simeq \bigl[V_L(n(\chi_{w_1}
-\chi_1))^*\otimes V_L(n\chi_{w_2})^*\otimes \dots \otimes
 V_L(n\chi_{w_s})^*\bigr]^L.$$
 Moreover, by Proposition ~\ref{prop6.1},
 $$H^0(\mathcal{Y},\mathcal{O}(nR)_{|\mathcal{Y}})\hookrightarrow
 H^0(\mathcal{Z},\mathcal{O}(m(n)R)), \text{for some}\, \,m(n)>0.$$
 Finally, by Corollary ~\ref{cor4.3}, for any $m\geq 1$,
 $$h^0(\mathcal{Z},\mathcal{O}(mR))=1.$$
 But, since the constants belong to $H^0(\mathcal{Y},\mathcal{O}(nR)_{|\mathcal{Y}})$, we have
 $$\dim (H^0(\mathcal{Y},\mathcal{O}(nR)_{|\mathcal{Y}})^G)\geq 1.$$
 This proves the identity ~\eqref{7.6} since
 ($\chi_1$ being a trivial character on the maximal torus of $L^{ss}$)
 $$\bigl[V_L(n(\chi_{w_1}
-\chi_1))^*\otimes V_L(n\chi_{w_2})^*\otimes \dots \otimes
 V_L(n\chi_{w_s})^*\bigr]^L \simeq
 \Hom_L(V_L(n\chi_1), V_L(n\chi_{w_1})\otimes \dots \otimes V_L(n\chi_{w_s})).$$

 The equivalence of ~\eqref{7.6} with ~\eqref{7.7} follows from [BK, Theorem 15].
 \end{proof}

 \begin{example} \label{example} (1) The converse to the above theorem is false in general. For example,
consider $G=\Sp(2\ell)$, $G/P$ the Lagrangian Grassmannian $LG(\ell,2\ell)$. It is cominuscule, so
the structure constants for the singular cohomology and the deformed cohomology
$\odot_0$  are the same.
The cells
in $LG(\ell,2\ell)$ are parametrized by the strict partitions
${\bf a}: (a_1 > a_2 > \dots > a_r > 0)$ and $a_1\leq \ell, r\leq \ell$ (cf. [FP, Page 29]).

The corresponding Levi subgroup is $GL(\ell)$, so the  Fulton conjecture (Theorem
~\ref{1.1}) holds.
 Now, take
$\ell=3$
and consider the cells in $LG(3,6)$ corresponding to the strict partitions $(1), (2>1), (2).$
 The corresponding
intersection number is $2$. The corresponding representations of the Levi subgroup have Young
diagrams $(2\geq 0\geq 0), (3\geq 3\geq 0)$ and $(3\geq 1\geq 0)$ respectively. Hence,
the dimension of
the invariant subspace for the corresponding tensor product of the Levi is $1$.

(2) In the above example, the intersection number is strictly larger than the
dimension of the invariant subspace for the corresponding tensor product. We also have examples where
the intersection number is strictly smaller than the
dimension of the invariant subspace for the corresponding tensor product. Take, for $G/P$ the
 Lagrangian Grassmannian $LG(5,10)$
and consider the cells  corresponding to the strict partitions $(3>1), (3>2), (4>2).$
 The
intersection number is $4$. The corresponding representations of the Levi subgroup have Young
diagrams $(4\geq 3\geq 1\geq 0\geq 0), (4\geq 4\geq 2\geq 0\geq 0)$ and
 $(5\geq 4\geq 2\geq 1\geq 0)$ respectively. Hence,
the dimension of
the invariant subspace for the corresponding tensor product of the Levi is $5$.

(3) Following the convention in [Bo], for $L$ of type $G_2$,
$[V(6\omega_1)\otimes V(6\omega_2)\otimes V(7\omega_2)]^L=1$,
and $[V(12\omega_1)\otimes V(12\omega_2)\otimes V(14\omega_2)]^L=2$. Similarly,
$[V(6\omega_1)\otimes V(6\omega_2)\otimes V(10\omega_1+\omega_2)]^L=1$ and
$[V(12\omega_1)\otimes V(12\omega_2)\otimes V(20\omega_1+2\omega_2)]^L=3,$ where
$\{\omega_1,\omega_2\}$ are the fundamental weights. Thus, the
 direct generalization of the Fulton's
 conjecture is false for general semisimple $L$.

 (4) There are examples of $w_1,w_2,w_3\in W^P$ such that
 \beqn \label{7.8}
 [X_{w_1}]\cdot [X_{w_2}]\cdot [X_{w_3}]=[X_e]
 \in H^{*}(G/P),
 \eeqn
 but ~\eqref{7.7} is false. Take, for example, $G=\Sp(6)$ and $P$ to be the maximal
 parabolic with $\Delta\setminus \Delta(P)=\{\alpha_2\}$ (following the convention in [Bo]).
 Now, take $w_1=w_2=s_1s_3s_2s_1s_3s_2,  w_3=s_3s_2.$ Then,
 ~\eqref{7.8} is satisfied (cf. [KLM, Theorem 4.6]). In this case, restricted to the
 Cartan of
  $L^{ss}$,
 we have $\chi_{w_1}=\chi_{w_2}=\omega_1+\omega_3, \chi_{w_3}=3\omega_1+\omega_3.$
 Thus, for any $n\geq 1$,
 $$\dim \bigl(\bigl[V_L(n\chi_{w_1})\otimes V_L(n\chi_{w_2}) \otimes V_L(n\chi_{w_3})\bigr]^{L^{ss}}\bigr)
 =0.$$
\end{example}
\begin{remark}\label{remarkable} {\em (1) If we specialize Theorem ~\ref{main} to $G=\GL(m)$ and $P$ any
 maximal parabolic subgroup, then (as explained in the introduction) we readily obtain a proof of Fulton's conjecture
 proved by Knutson-Tao-Woodward [KTW] (Belkale [B$_2$] and Ressayre [R$_2$] gave
 other geometric proofs)
 asserting the following:

 Let $V_L(\lambda_1), \dots, V_L(\lambda_s)$ be finite dimensional irreducible
 representations of $L=\GL(r)$ with highest weights $\lambda_1, \dots , \lambda_s$
 respectively. Assume that  $\bigl[V_L(\lambda_1)\otimes \dots \otimes
 V_L(\lambda_s)\bigr]^{L^{ss}}$ is one dimensional. Then, for any $n\geq 1$,
 $\bigl[V_L(n\lambda_1)\otimes \dots \otimes
 V_L(n\lambda_s)\bigr]^{L^{ss}}$ again is one dimensional.

 In fact, since any maximal parabolic subgroup in $GL(m)$ is cominuscule, by a result
 of Brion-Polo [BP], we have $\mathcal{Z}=\mathcal{Y}$. Hence,
 Proposition ~\ref{prop6.1} and the results from Section 7 are {\it not}
 needed in this case.

(2) We now specialize Theorem ~\ref{main}  to $G=\SP(2\ell)$ and
$G/P=\LG(\ell,2\ell)$  the Lagrangian Grassmannian.
Under the assumption  that some structure coefficient of
$(H^*(\LG(\ell,2\ell)),\odot_0)$ in the Schubert basis is equal to one,
the conclusion of the theorem
is that some Littlewood-Richardson coefficient is equal to one.
In \cite{RLG}, it is shown that this assumption
is fulfilled if and only if some Littlewood-Richardson coefficient is equal to one.
Hence by combining Theorem ~\ref{main}  and \cite{RLG}, we obtain the following
result on Littlewwod-Richardson coefficients.

Let $\lambda$, $\mu$ and $\nu$ be three partitions. We assume that the
 Young diagrams of $\lambda$, $\mu$ and $\nu$  are contained in
the square of size $\ell$ and are symmetric relative to the diagonal.
Then, for the Littlewood-Richardson coeffcients for $\GL(\ell)$,
$$
c_{\lambda,\,\mu}^\nu=1 \  \Rightarrow \ c_{\lambda',\,\mu'}^{\nu'}=1,
$$
 where $\lambda'$ and $\mu'$   are obtained from $\lambda$
 and $\mu$  by adding
  one to some initial parts (for the details, see [R$_3$]), and $\nu'$ is defined by dualizing
the rule.}
\end{remark}

\bibliographystyle{plain}
\def\noopsort#1{}

\vskip6ex

\noindent
Addresses:

\noindent
P.B.: Department of Mathematics, University of North Carolina, Chapel Hill, NC 27599-3250, USA\\
\noindent(email: belkale@email.unc.edu)

\noindent
S.K.: Department of Mathematics, University of North Carolina, Chapel Hill, NC 27599-3250, USA\\
\noindent(email: shrawan@email.unc.edu)

\noindent
N.R.: Department of Mathematics, Universit\'e Montpellier II, CC 51-Place Eug\`ene
Bataillon, 34095 Montpellier,
Cedex 5, France\\
\noindent(email: ressayre@math.univ-montp2.fr)

  \end{document}